\documentclass[11pt]{amsart}

\usepackage[T1]{fontenc}
\usepackage{lmodern}
\usepackage{microtype}
\usepackage{mathtools,amssymb,mathrsfs}
\usepackage[margin=1.15in]{geometry}
\usepackage[colorlinks=true,linkcolor=blue,citecolor=blue,urlcolor=blue]{hyperref}

\newtheorem{theorem}{Theorem}[section]
\newtheorem{proposition}[theorem]{Proposition}
\newtheorem{lemma}[theorem]{Lemma}
\newtheorem{corollary}[theorem]{Corollary}
\newtheorem{remark}[theorem]{Remark}

\newtheorem*{theoremA}{Theorem A}
\newtheorem*{theoremB}{Theorem B}
\newtheorem*{theoremC}{Theorem C}
\newtheorem*{theoremD}{Theorem D}

\newcommand{\Lthree}{\mathbb L^3}
\newcommand{\HH}{\mathbb H^2}
\newcommand{\dS}{\mathbb S^2_1}

\newcommand{\Hess}{\operatorname{Hess}}
\newcommand{\tr}{\operatorname{tr}}
\newcommand{\Id}{\operatorname{Id}}
\newcommand{\cl}{\overline}

\title[Overdetermined equations on $\HH$ and $\dS$]
{Overdetermined equations and support functions on the hyperbolic and de Sitter planes}

\author{M\'arcio Batista}
\author{Iury Domingos}
\address{CPMAT -- Instituto de Matem\'atica, Universidade Federal de Alagoas, Macei\'o, AL, 57072-970, Brazil}
\email{mhbs@mat.ufal.br}
\address{Universidade Federal de Alagoas, Av. Manoel Severino Barbosa S/N, Arapiraca, AL, 57309-005, Brazil}

\email{iury.domingos@arapiraca.ufal.br}

\subjclass[2020]{Primary 35N25, 53C50; Secondary 35P15, 53A10, 53C42}
\keywords{Overdetermined problem, support function, hyperbolic plane, de Sitter plane, maximal surface, timelike minimal surface}

\begin{document}

\begin{abstract}
We study overdetermined equations on the hyperbolic and de Sitter planes by
means of the support-function representation of zero mean curvature surfaces
in Lorentz--Minkowski three-space. On the hyperbolic plane, solutions of
$\Delta u-2u=0$ give rise to branched spacelike maximal surfaces. We prove
that, on a bounded simply connected domain, nontrivial constant Dirichlet and
Neumann data force the domain to be a geodesic disk, provided that the
associated support quadric is nonlightlike; moreover, the solution is a
multiple of the hyperbolic cosine of the distance from the center. On the de
Sitter plane, the corresponding equation is the Klein--Gordon equation
$\Box u+2u=0$, and its solutions generate timelike minimal surfaces wherever
the support tensor is nondegenerate. We identify the support quadric and the
constant-angle condition determined by constant Cauchy data, prove that local
rigidity fails along every analytic noncharacteristic curve, and recover
rotational symmetry from constant data on a global Cauchy circle. Thus the
same support-function formalism reveals a sharp transition from elliptic
rigidity to hyperbolic flexibility when the signature changes.
\end{abstract}

\maketitle

\section{Introduction}

Overdetermined boundary value problems lie at the intersection of partial
differential equations and geometric rigidity. Their classical prototype is
Serrin's theorem: the simultaneous prescription of constant Dirichlet and
Neumann data for a solution of a semilinear elliptic equation forces a bounded
Euclidean domain to be a ball \cite{Serrin1971}. On space forms, the geometry of the ambient manifold plays an essential role. Moving-plane arguments extend
Serrin-type rigidity to hyperbolic space and to domains contained in a
hemisphere \cite{KumaresanPrajapat1998}, while geometric representations of
the equation provide a different and particularly effective approach for
distinguished curvature-dependent parameters.

For domains in the round two-sphere, Souam \cite{Souam2005} exploited such a
representation for the equation
\[
 \Delta_{\mathbb S^2}u+2u=0.
\]
The map $X_u(p)=u(p)p+\nabla u(p)$ is, away from its branch points, a minimal
immersion in $\mathbb R^3$ whose Gauss map is the inclusion. Constant
Dirichlet and Neumann data place its boundary on a round sphere and impose a
constant contact angle. The capillary rigidity of minimal disks then forces
total umbilicity and, consequently, shows that the original domain is a geodesic disk.
This argument suggests that overdetermination can be understood through the
geometry of the surface encoded by the support function.

The support-function construction itself belongs to a broader framework.
Reznikov \cite{Reznikov1992} observed that, when the Gauss map is locally
invertible, the support function linearizes the zero mean curvature equation
in dimension two. In Lorentz--Minkowski space, the Gauss map of a spacelike
maximal surface takes values in the hyperbolic plane, whereas that of a timelike minimal surface lies in the de Sitter plane. These two
targets have the same constant-curvature origin but opposite metric
signatures. Accordingly, the support equation is elliptic on $\HH$ and
hyperbolic on $\dS$. The purpose of this paper is to develop these two theories
in parallel and to determine precisely how the change of signature affects
the associated overdetermined problem.

Our first result provides the support representation in a unified form. If
$Q_\varepsilon^2\subset\Lthree$ is the quadric
$\langle p,p\rangle=\varepsilon$ and
\[
 X_u(p)=\varepsilon u(p)p+\nabla u(p),
\]
then
\[
 dX_u=\Hess u+\varepsilon u\Id.
\]
Hence, on the regular set, the equation
$\Delta u+2\varepsilon u=0$ is equivalent to the vanishing of the mean
curvature of the associated surface. In the Riemannian case, the trace-free
support tensor determines a holomorphic quadratic differential; consequently, the support map is either constant or has only isolated branch points. This yields
the following representation theorem on the hyperbolic plane.

\begin{theoremA}\label{thm:intro-H-representation}
Let $\Omega\subset\HH$ be a domain and let $u\in C^2(\Omega)$ satisfy
\[
 \Delta_{\HH}u-2u=0.
\]
Then $X_u=-up+\nabla u$ is either constant or a branched spacelike maximal
immersion with isolated branch points. On its regular set, its Gauss map is the inclusion into $\HH$. Conversely, every spacelike maximal surface whose Gauss map is locally invertible admits locally such an inverse Gauss parametrization, and its support function satisfies the preceding equation.
\end{theoremA}

The statements of this introduction are restated in full in the body of the paper. The central elliptic result is the following hyperbolic
counterpart of Souam's rigidity theorem.

\begin{theoremB}\label{thm:intro-H-rigidity}
Let $\Omega\subset\HH$ be a bounded simply connected $C^{2,\alpha}$ domain.
Suppose that a nonzero function
$u\in C^2(\Omega)\cap C^1(\overline\Omega)$ satisfies
\[
 \Delta_{\HH}u-2u=0\quad\text{in}\quad\Omega,
 \qquad u=a,\quad \partial_\nu u=b
 \quad\text{on}\quad\partial\Omega,
\]
where $a^2\ne b^2$. Then $\Omega$ is a geodesic disk and
\[
 u=A\cosh r,
\]
where $r$ is the hyperbolic distance from its center. In particular,
$a^2-b^2=A^2>0$.
\end{theoremB}

Indeed, the boundary data imply
\[
 \langle X_u,X_u\rangle=b^2-a^2,
 \qquad \langle X_u,p\rangle=a
 \quad\text{on}\quad\partial\Omega.
\]
Thus, when $a^2\ne b^2$, the associated maximal disk has boundary on a
nondegenerate totally umbilical quadric and meets it at constant Lorentzian
angle. The capillary rigidity results of Al\'ias--Pastor
\cite{AliasPastor1999} and Pyo--Seo \cite{PyoSeo2011}, together with the Hopf-differential argument, force the disk to be planar. In inverse Gauss
coordinates, this means that the support map is constant, from which the
geodesic-disk conclusion follows. The excluded relation $a^2=b^2$ is not a
technical artifact: it places the boundary on the light cone, where the
standard nondegenerate capillary theory is unavailable.

We next turn to the de Sitter plane
\[
 \dS=\{p\in\Lthree:\langle p,p\rangle=1\}.
\]
In this setting, the same construction becomes
\[
 X_u(p)=u(p)p+\nabla u(p),
 \qquad dX_u=\Hess u+u\Id,
\]
and zero mean curvature is equivalent to the Klein--Gordon equation
\[
 \Box_{\dS}u+2u=0.
\]
This gives an inverse Gauss-map description of timelike minimal surfaces,
complementary to the Lorentz-conformal and split-complex approaches developed,
for example, by Magid \cite{Magid1991}, Inoguchi \cite{Inoguchi1998}, and
Inoguchi--Toda \cite{InoguchiToda2004}.

The change from an elliptic to a hyperbolic equation fundamentally alters the
rigidity problem. Along a noncharacteristic curve, the prescribed values of
$u$ and its conormal derivative are Cauchy data. We obtain the following local
flexibility result.

\begin{theoremC}\label{thm:intro-dS-flexibility}
Let $\Gamma\subset\dS$ be an analytic noncharacteristic curve and let
$a,b\in\mathbb R$. Then there exists a unique analytic solution of
\[
 \Box_{\dS}u+2u=0,
 \qquad u|_\Gamma=a,
 \qquad \partial_\nu u|_\Gamma=b
\]
in a neighborhood of $\Gamma$. Thus, constant Cauchy data impose no
local restriction on the geometry of $\Gamma$.
\end{theoremC}

Moreover, if $\langle\nu,\nu\rangle=\delta\in\{-1,1\}$, then the associated
boundary lies on the quadric
\[
 \mathcal Q_\rho=\{x\in\Lthree:\langle x,x\rangle=\rho\},
 \qquad \rho=a^2+\delta b^2,
\]
and, whenever $\rho\ne0$ and the support tensor is nondegenerate, the timelike
minimal surface meets $\mathcal Q_\rho$ at constant Lorentzian angle. Hence
the geometric capillary condition persists, but it no longer selects the
boundary curve.

Although local rigidity fails, symmetry is recovered when the data are
prescribed globally on an invariant Cauchy curve.

\begin{theoremD}\label{thm:intro-dS-rotational}
Let $u$ solve $\Box_{\dS}u+2u=0$ on a globally hyperbolic strip containing the
Cauchy circle $\Sigma_{t_0}=\{t=t_0\}$. If $u$ and $u_t$ are constant on
$\Sigma_{t_0}$, then $u$ is rotationally invariant. More precisely,
$u(t,\theta)=f(t)$, where
\[
 f''+\tanh t\,f'-2f=0.
\]
\end{theoremD}

The proof follows directly from uniqueness for the global Cauchy problem:
rotations preserve both the equation and the prescribed data. Equivalently,
the Fourier decomposition shows that all nonzero angular modes vanish. Taken
together, Theorems~A--D exhibit the main dichotomy of the paper: elliptic
overdetermination on $\HH$ produces domain rigidity, whereas hyperbolic Cauchy
data on $\dS$ allow arbitrary local boundary geometry, with symmetry returning
only under a global invariant prescription.

\textbf{The paper is organized as follows.} In Section~\ref{sec:ambient}, we derive the unified support-function formula
and record the holomorphicity of the support tensor in the Riemannian setting.
In Section~\ref{sec:hyperbolic}, we develop the maximal-surface representation
and prove the nonlightlike rigidity theorem on $\HH$. Finally, in
Section~\ref{sec:desitter}, we study the timelike minimal representation,
boundary quadrics, local flexibility, and global rotational rigidity on $\dS$.

\section{A unified support-function calculation}\label{sec:ambient}

Let $\Lthree=(\mathbb R^3,\langle\,,\,\rangle)$ have signature $(+,+,-)$ and
define
\[
 Q_\varepsilon^2=\{p\in\Lthree:\langle p,p\rangle=\varepsilon\},
 \quad \varepsilon\in\{-1,1\}.
\]
We use the upper sheet for $Q_{-1}^2=\HH$ and write $Q_1^2=\dS$.  The induced
metric $g_\varepsilon$ is Riemannian for $\varepsilon=-1$ and Lorentzian for
$\varepsilon=1$.  Its Laplace--Beltrami operator is denoted by $\Delta$ in the
Riemannian case and by $\Box$ in the Lorentzian case.  Both are defined as the
metric trace of the Hessian.  With this convention,
\[
 \Delta_{\HH}\cosh r=2\cosh r,
\]
where \(r\) denotes the hyperbolic distance from a fixed point.

The ambient connection $D$ of $\Lthree$ and the Levi--Civita connection
$\nabla$ of $Q_\varepsilon^2$ are related by the Gauss formula
\begin{equation}\label{eq:gauss-quadric}
	D_YZ=\nabla_YZ-\varepsilon\langle Y,Z\rangle p.
\end{equation}

\begin{proposition}\label{prop:support-unified}
Let $u\in C^2(\Omega)$, where $\Omega\subset Q_\varepsilon^2$ is a domain, and define
\begin{equation}\label{eq:support-map}
 X_u(p)=\varepsilon u(p)p+\nabla u(p).
\end{equation}
Then
\begin{equation*}\label{eq:support-differential}
 dX_u=\Hess u+\varepsilon u\Id.
\end{equation*}
In particular, $dX_u(T_pQ_\varepsilon^2)\subset T_pQ_\varepsilon^2$ and
$\langle X_u,p\rangle=u$.
\end{proposition}

\begin{proof}
For $Y\in T_pQ_\varepsilon^2$, formula \eqref{eq:gauss-quadric} gives
\[
 D_Y\nabla u=\nabla_Y\nabla u-\varepsilon Y(u)p.
\]
Differentiating \eqref{eq:support-map}, we find that the two normal terms cancel, so that
\[
 D_YX_u=\Hess u(Y)+\varepsilon uY.
\]
The last assertion follows from $\langle p,p\rangle=\varepsilon$ and
$\langle\nabla u,p\rangle=0$.
\end{proof}

We set
\[
 B_u=\Hess u+\varepsilon u\Id.
\]
Where $B_u$ is invertible, $X_u$ is an immersion and its Gauss map is $p$ up
to orientation.  With the convention $A=dG$, its shape operator is
$A=B_u^{-1}$.

\begin{lemma}\label{lem:trace-inverse}
For every invertible endomorphism $B$ of a two-dimensional vector space,
\[
 \tr(B^{-1})=\frac{\tr B}{\det B}.
\]
Hence, at every regular point of $X_u$,
\[
 H_{X_u}=0\quad\text{if and only if}\quad
 \tr B_u=0.
\]
\end{lemma}

This elementary identity is independent of the signature.  It yields the
unified equation
\begin{equation*}\label{eq:unified-pde}
 \operatorname{tr}_{g_\varepsilon}\Hess u+2\varepsilon u=0.
\end{equation*}
For $\varepsilon=-1$, this is $\Delta u-2u=0$; for $\varepsilon=1$, it is
$\Box u+2u=0$. Moreover, as we prove below, every solution of \(\Delta u-2u=0\) gives rise to a holomorphic quadratic differential associated with \(B_u\), with respect to the conformal structure of \(\mathbb{H}^2\).
 
\begin{lemma}
Let $(\Omega,g)$ be a domain in an oriented Riemannian surface of constant Gaussian
curvature $\varepsilon$. Suppose that
\[
\Delta u+2\varepsilon u=0
\]
and set
\(
B_u=\Hess u+\varepsilon u\,g.
\)
Then the $(2,0)$-part of $B_u$ is a holomorphic quadratic differential.
\end{lemma}

\begin{proof}
If $z$ is a conformal coordinate such that
\(
g=e^{2\lambda}|dz|^2,
\)
then
\(
(B_u)^{2,0}
=
\left(u_{zz}-2\lambda_z u_z\right)dz^2.
\)
Set
\(
q=u_{zz}-2\lambda_z u_z.
\)
Since
\(
\Delta u=4e^{-2\lambda}u_{z\bar z},
\)
the equation $\Delta u+2\varepsilon u=0$ gives
\[
u_{z\bar z}
=-\frac{\varepsilon}{2}e^{2\lambda}u.
\]
Hence,
\[
\begin{aligned}
q_{\bar z}
&=u_{zz\bar z}
  -2\lambda_{z\bar z}u_z
  -2\lambda_z u_{z\bar z}=-\left(
  \frac{\varepsilon}{2}e^{2\lambda}
  +2\lambda_{z\bar z}
  \right)u_z.
\end{aligned}
\]
The curvature identity
\(
K=-4e^{-2\lambda}\lambda_{z\bar z}
\)
therefore yields
\[
q_{\bar z}
=\frac{e^{2\lambda}}2(K-\varepsilon)u_z.
\]
Since $K=\varepsilon$, we conclude that $q_{\bar z}=0$.
\end{proof}

A direct consequence is the following result.

\begin{corollary}
Under the assumptions of the previous lemma, either $B_u\equiv0$ or the
zeros of $B_u$ are isolated.
\end{corollary}

\begin{proof}
Since $B_u$ is trace-free, it vanishes at a point if and only if its
$(2,0)$-part vanishes there. The latter is a holomorphic quadratic
differential, whose zeros are therefore isolated unless it vanishes identically.
\end{proof}

\section{The hyperbolic plane}\label{sec:hyperbolic}

\subsection{Maximal surfaces and the equation \texorpdfstring{$\Delta u-2u=0$}{Delta u - 2u = 0}}

For a domain $\Omega\subset\HH$, define
\begin{equation}\label{eq:X-hyperbolic}
 X_u(p)=-u(p)p+\nabla u(p).
\end{equation}

We next provide a geometric interpretation of the solutions of a linear elliptic PDE. The result reads as follows.

\begin{theorem}
\label{thm:H-representation}
Let $\Omega\subset\HH$ be a domain and let $u\in C^2(\Omega)$ satisfy
\begin{equation}\label{eq:H-equation}
 \Delta_{\HH}u-2u=0.
\end{equation}
Then the map $X_u$ defined by~\eqref{eq:X-hyperbolic} is either constant or
a branched spacelike maximal immersion whose branch points are isolated.
More precisely, its regular set is
\[
 \Omega_{\rm reg}
 =
 \bigl\{p\in\Omega:
 \det(\Hess u-u\Id)\neq0
 \bigr\},
\]
and $\Omega\setminus\Omega_{\rm reg}$ is discrete. On
$\Omega_{\rm reg}$, the Gauss map of $X_u$ is the inclusion
\(
 \Omega_{\rm reg}\hookrightarrow\HH.
\)

Conversely, let $X$ be a spacelike maximal surface in $\Lthree$ whose
Gauss map is locally invertible away from its branch points. Then its inverse
Gauss parametrization is locally of the form~\eqref{eq:X-hyperbolic}, and
the corresponding support function satisfies~\eqref{eq:H-equation}.
\end{theorem}
\begin{proof}
By Proposition~\ref{prop:support-unified}, the differential of the
map
\(
X_u(p)=-u(p)p+\nabla u(p)
\)
is given by
\begin{equation}\label{eq:dXu-hyperbolic}
 dX_u(Y)=\Hess u(Y)-uY
\end{equation}
for every \(Y\in T_p\HH\).  Set
\(
B_u=\Hess u-u\Id.
\)
Since \(B_u\) is self-adjoint with respect to the hyperbolic metric, the
regular set of \(X_u\) is
\(
\Omega_{\mathrm{reg}}
 =
 \{p\in\Omega:\det B_u(p)\neq0\}.
\)
Indeed, \(dX_u=B_u\), so \(X_u\) is an immersion precisely where \(B_u\)
is invertible.

For \(p\in\Omega_{\mathrm{reg}}\) and \(Y,Z\in T_p\HH\), the metric
induced by \(X_u\) is
\[
(X_u^*\langle\,,\,\rangle)(Y,Z)
 =
 \langle B_uY,B_uZ\rangle.
\]
The restriction of the Lorentzian metric to \(T_p\HH=p^\perp\) is
positive definite.  Since \(B_u\) is invertible on
\(\Omega_{\mathrm{reg}}\), the preceding identity shows that the induced
metric is positive definite.  Hence \(X_u\) is a spacelike immersion
there.

Moreover, equation~\eqref{eq:dXu-hyperbolic} implies
\(
\langle dX_u(Y),p\rangle=0,
\)
because \(B_uY\in T_p\HH=p^\perp\).  Thus \(p\) is a timelike unit normal
to the immersed surface at \(X_u(p)\).  Hence, up to the choice
of time orientation, the Gauss map of \(X_u\) is
\[
G(X_u(p))=p.
\]

We use the convention that the shape operator is \(A=dG\).  Since
\(G\circ X_u=\Id_{\Omega_{\mathrm{reg}}}\), differentiation gives
\(
dG_{X_u(p)}\circ dX_u|_p=\Id_{T_p\HH}.
\)
Therefore
\[
A\circ B_u=\Id
\quad\text{and hence}\quad
A=B_u^{-1}.
\]
Since \(B_u\) is an invertible endomorphism of a two-dimensional vector
space, Lemma~\ref{lem:trace-inverse} yields
\(\
\tr A
 =
 \tr(B_u^{-1})
 =
 \frac{\tr B_u}{\det B_u}.
\)
On the other hand,
\[
\tr B_u
 =
 \tr(\Hess u-u\Id)
 =
 \Delta_{\HH}u-2u.
\]
It follows from~\eqref{eq:H-equation} that
\(
\tr A=0.
\)
Thus the mean curvature of \(X_u\) vanishes, and \(X_u\) is a spacelike
maximal immersion on \(\Omega_{\mathrm{reg}}\).

We now prove the converse.  Let
\(
X\colon M\rightarrow\Lthree
\)
be a spacelike maximal immersion whose Gauss map
\(
G\colon M\rightarrow\HH
\)
is a local diffeomorphism.  After restricting to a sufficiently small
neighborhood, we may use the Gauss map as a parametrization and regard
\(X\) as a map
\(
X\colon\Omega\subset\HH\rightarrow\Lthree
\)
satisfying
\(
G(X(p))=p.
\)
Define the support function by
\(
u(p)=\langle X(p),p\rangle.
\)

At each \(p\in\Omega\), decompose \(X(p)\) into its normal and tangential
components relative to \(\HH\).  Since
\(\langle p,p\rangle=-1\), the normal component is
\(
X^\perp=-\langle X,p\rangle p=-up.
\)
We claim that the tangential component is \(\nabla u\).  Indeed, for
\(Y\in T_p\HH\), differentiation of \(u=\langle X,p\rangle\) gives
\(
Y(u)
 =
 \langle dX(Y),p\rangle+\langle X,Y\rangle.
\)
Because \(p\) is the Gauss normal of the surface,
\(dX(Y)\perp p\), and hence
\(
Y(u)=\langle X,Y\rangle
     =\langle X^\top,Y\rangle.
\)  
Therefore
\(
X^\top=\nabla u,
\)
and so
\begin{equation}\label{eq:inverse-gauss-support}
 X=-up+\nabla u.
\end{equation}

Applying Proposition~\ref{prop:support-unified}, or differentiating
\eqref{eq:inverse-gauss-support} directly, we obtain
\[
dX=\Hess u-u\Id=B_u.
\]
Since \(X\) is an immersion and \(G\) is a local diffeomorphism, \(B_u\)
is invertible.  As above, its shape operator is
\(
A=B_u^{-1}.
\)
The maximality condition gives
\[
0=\tr A=\tr(B_u^{-1})
       =\frac{\tr B_u}{\det B_u}.
\]
Since \(\det B_u\neq0\), we conclude that
\(
0=\tr B_u=\Delta_{\HH}u-2u.
\)
Thus the support function satisfies
\[
\Delta_{\HH}u-2u=0,
\]
and the inverse Gauss parametrization is precisely
\(
X(p)=-u(p)p+\nabla u(p).
\)
\end{proof}

\subsection{Constant boundary data and Lorentzian capillarity}

Let the domain $\Omega\subset\HH$ have smooth boundary, and let $\nu$ be its exterior
unit conormal.  Suppose
\begin{equation}\label{eq:H-boundary-data}
 u=a,\quad \partial_\nu u=b
 \quad\text{on}\quad\partial\Omega.
\end{equation}
Since the tangential derivative of the solution vanishes, $\nabla u=b\nu$ on the boundary.
Thus
\begin{equation}\label{eq:H-boundary-norm}
 \langle X_u,X_u\rangle=b^2-a^2,
 \quad \langle X_u,p\rangle=u=a.
\end{equation}

If $b^2-a^2>0$, the boundary lies on a de Sitter quadric.  If
$b^2-a^2<0$, it lies on a hyperbolic quadric.  Equality corresponds to the
light cone.  In either nondegenerate case, the second identity in
\eqref{eq:H-boundary-norm} shows that the maximal surface meets the support
quadric at constant Lorentzian angle.

\begin{proposition}\label{prop:H-capillary}
Assume $a^2\neq b^2$.  A solution of \eqref{eq:H-equation} and
\eqref{eq:H-boundary-data} whose support map is regular along the boundary
generates a spacelike maximal surface with boundary on a nondegenerate totally
umbilical quadric centered at the origin, meeting that quadric at constant
contact angle.
\end{proposition}

Next, we establish the following Souam-type rigidity result in the hyperbolic setting.

\begin{theorem}\label{thm:H-rigidity}
Let $\Omega\subset\HH$ be a bounded simply connected $C^{2,\alpha}$ domain.
Suppose that a nonzero function
$u\in C^2(\Omega)\cap C^1(\cl\Omega)$ satisfies
\begin{equation*}\label{eq:H-overdetermined}
 \Delta u-2u=0\quad\text{in}\quad\Omega,
 \qquad u=a,\quad \partial_\nu u=b
 \quad\text{on}\quad\partial\Omega,
\end{equation*}
with $a^2\neq b^2$.  Then $\Omega$ is a geodesic disk and
\(
 u=A\cosh r,
\)
where $r$ is the distance from its center.  In particular,
$a^2-b^2=A^2>0$.
\end{theorem}
\begin{proof}
Consider the support map
\(
X_u\colon\Omega\rightarrow\Lthree,
\,
X_u(p)=-u(p)p+\nabla u(p).
\)
By Theorem~\ref{thm:H-representation}, its restriction to the regular set
\(
\Omega_{\mathrm{reg}}
 =
 \left\{p\in\Omega:
 \det\bigl(\Hess u-u\Id\bigr)(p)\neq0\right\}
\)
is a spacelike maximal immersion whose Gauss map is, up to the choice of time
orientation, the inclusion
\(
G(X_u(p))=p.
\)
By hypothesis, the singularities of \(X_u\) are isolated unless \(X_u\) is
constant. Since \(\Omega\) is simply connected, \(X_u\) may therefore be
regarded as a branched spacelike maximal disk.

Along \(\partial\Omega\), the boundary conditions give
\(
u=a,
\,
\nabla u=b\nu,
\)
and hence
\[
\langle X_u,X_u\rangle=b^2-a^2,
\quad
\langle X_u,p\rangle=a.
\]
Since \(a^2\neq b^2\), the boundary of the associated maximal disk lies on a
nondegenerate totally umbilical quadric of \(\Lthree\). Moreover, the second
identity shows that the maximal disk meets this quadric at a constant
Lorentzian contact angle; see
Proposition~\ref{prop:H-capillary}. The Lorentzian
Terquem--Joachimsthal theorem then implies that every regular component of the
boundary is a line of curvature of the maximal surface.

We next show that \(X_u\) is totally umbilical. In the regular case, one may directly invoke the
disk-type capillary rigidity theorem of Alías--Pastor
\cite{AliasPastor1999}, or its extension by Pyo--Seo
\cite{PyoSeo2011}. For completeness, we recall the Hopf-differential
argument, which also accommodates isolated branch points.

Let
\(
X\colon\overline{\mathbb D}\rightarrow\Lthree
\)
be a conformal parametrization of the branched maximal disk, and let
\(
\Phi(z)\,dz^2
\)
denote its Hopf differential. Since \(X\) has constant mean curvature,
\(\Phi\) is holomorphic on the regular set. It extends holomorphically across
the isolated branch points. Because the boundary is a line of curvature, the
mixed coefficient of the second fundamental form vanishes along
\(\partial\mathbb D\). Equivalently,
\[
\operatorname{Im}\bigl(z^2\Phi(z)\bigr)=0
\quad\text{on}\quad\partial\mathbb D.
\]
The function
\(
\operatorname{Im}\bigl(z^2\Phi(z)\bigr)
\)
is harmonic on \(\mathbb D\) and continuous on \(\overline{\mathbb D}\).
The maximum principle therefore gives
\(
\operatorname{Im}\bigl(z^2\Phi(z)\bigr)=0
\) in \(\mathbb{D}\).
Hence the holomorphic function \(z^2\Phi(z)\) is real-valued and must be
constant. Since it vanishes at \(z=0\), it follows that
\[
z^2\Phi(z)\equiv0,
\quad\text{and therefore}\quad
\Phi\equiv0.
\]
Thus the maximal disk is totally umbilical. Since its mean curvature vanishes,
its second fundamental form vanishes identically on the regular set.
Consequently, every regular component of \(X_u\) is contained in a spacelike
affine plane and has constant Gauss map.

This conclusion is incompatible with the inverse Gauss parametrization unless
\(X_u\) is constant. Indeed, if \(X_u\) possessed a regular point, then on a
neighborhood of that point its Gauss map would simultaneously be constant,
because the image is planar, and equal to the identity
by the construction of the support map. This is impossible on an open subset
of \(\HH\). Hence \(X_u\) has no regular points. Since the singular points are assumed to be isolated unless the map is constant, we conclude that
\(
X_u\equiv q
\)
for some fixed vector \(q\in\Lthree\).

Using the support identity
\(
u(p)=\langle X_u(p),p\rangle,
\)
we obtain
\begin{equation}\label{eq:u-linear-hyperbolic}
u(p)=\langle q,p\rangle.
\end{equation}
The vector \(q\) is nonzero because \(u\) is nontrivial. Moreover,
\(\partial\Omega\) is contained in the level set
\[
\left\{p\in\HH:\langle q,p\rangle=a\right\}.
\]
Since \(\partial\Omega\) is compact, \(q\) must be timelike. Indeed, the level
sets associated with spacelike vectors are equidistant curves, while those
associated with lightlike vectors are horocycles; neither type is compact.
Only a timelike vector yields a geodesic circle.

After applying a Lorentz isometry preserving \(\HH\), we may therefore write
\[
q=-Ae_3,
\quad
A\neq0,
\]
where \(e_3=(0,0,1)\) is the future-directed unit timelike vector. If \(r\)
denotes the hyperbolic distance from the point \(e_3\in\HH\), then
\(
-\langle p,e_3\rangle=\cosh r.
\)
Equation~\eqref{eq:u-linear-hyperbolic} becomes
\[
u(p)=A\cosh r,
\]
where
\(
r(p)=d_{\mathbb H^2}(p,e_3).
\)
Since \(u=a\) along \(\partial\Omega\), it follows that
\(
A\cosh r(p)=a,
\, p\in\partial\Omega.
\)
As \(A\neq0\) and \(\cosh r\) is strictly increasing for \(r\geq0\), there exists a unique \(R>0\) such that
\(
\cosh R=\frac{a}{A},
\)
and hence
\[
\partial\Omega\subset C_R=\{p\in\mathbb H^2:r(p)=R\}.
\]
Since
\(\partial\Omega\) is a compact regular closed curve contained in the
one-dimensional connected manifold \(C_R\), it is both open and closed in
\(C_R\). Therefore,
\(
\partial\Omega=C_R.
\)

The geodesic circle \(C_R\) separates \(\mathbb H^2\) into exactly two
connected components,
\(
\{r<R\}\)
and \linebreak
\(\{r>R\}.
\)
Since \(\Omega\) is connected and bounded and
\(\partial\Omega=C_R\), it follows that
\[
\Omega=\{p\in\mathbb H^2:r(p)<R\},
\]
 that is, \(\Omega\) is the geodesic disk of radius \(R\) centered at \(e_3\).

Finally, along \(\partial\Omega\), the exterior unit conormal is
\(\nu=\partial_r\). Therefore
\[
a=A\cosh R,
\quad
b=\partial_\nu u=A\sinh R.
\]
It follows that
\[
a^2-b^2
 =
A^2\bigl(\cosh^2R-\sinh^2R\bigr)
 =
A^2>0.
\]
This proves both the rigidity of the domain and the asserted form of the
solution.
\end{proof}

\begin{remark}\label{rem:lightlike-H}
When $a^2=b^2$, the image of the boundary lies on the light cone.  The light cone is degenerate and is neither a spacelike nor a timelike totally umbilical support surface in the sense required by the standard capillary theorem.  The proof
above therefore does not cover this case.  Note that the geodesic-disk model never satisfies $a^2=b^2$ unless $u\equiv0$.
\end{remark}

\begin{remark}
Theorem~\ref{thm:H-rigidity} is also related to the Hopf-type framework of
Espinar--Mazet~\cite{EM19}. Indeed, after changing the sign of \(u\), if
necessary, and setting \(w=a-u\), the problem becomes
\[
\Delta_{\HH}w+2a-2w=0,
\quad
w=0,\quad \partial_\nu w=-b
\quad\text{on}\quad\partial\Omega.
\]
For the case \(a^2>b^2\), the method of Espinar--Mazet provides an alternative
Hopf-type interpretation of the disk rigidity. Our proof follows a different
geometric approach and also includes the spacelike case \(a^2<b^2\), which is
ruled out by the Lorentzian geometry of the associated support-surface.
\end{remark}

\section{The de Sitter plane}\label{sec:desitter}

\subsection{The support equation and timelike minimal surfaces}

For a domain $\Omega\subset\dS$ and $u\in C^2(\Omega)$, set
\begin{equation}\label{eq:X-dS}
 X_u(p)=u(p)p+\nabla u(p)
\end{equation}
and \(B_u=\Hess u+u\Id.\) In the following, we provide a geometric interpretation of the solutions of a linear hyperbolic PDE.

\begin{theorem}\label{thm:dS-representation}
Let $\Omega\subset\dS$ be a domain and let $u\in C^2(\Omega)$ satisfy
\begin{equation}\label{eq:dS-equation}
 \Box_{\dS}u+2u=0.
\end{equation}
Then, on the regular set
\(
\Omega_{\rm reg}
=
\bigl\{p\in\Omega:
\det(\Hess u+u\Id)\neq0
\bigr\},
\)
the map $X_u$ defined by~\eqref{eq:X-dS}  is a timelike
minimal immersion in $\Lthree$ and its spacelike Gauss map is the inclusion
\(
\Omega_{\rm reg}\hookrightarrow\dS.
\)
Conversely, let $X$ be a timelike minimal surface in $\Lthree$ whose
Gauss map is locally invertible away from its branch points. Then its inverse
Gauss parametrization is locally of the form~\eqref{eq:X-dS}, and
the corresponding support function satisfies~\eqref{eq:dS-equation}.
\end{theorem}

\begin{proof}
By Proposition~\ref{prop:support-unified},
\(
dX_u=B_u:=\Hess u+u\Id.
\)
Hence, on the open set where \(B_u\) is invertible, \(X_u\) is a timelike
immersion and \(p\in\dS\) is its unit spacelike normal. Since
\(G\circ X_u=\Id\), where \(G\) denotes the Gauss map, the shape operator is
\(
A=B_u^{-1}.
\)
Lemma~\ref{lem:trace-inverse} then gives
\[
2H=\tr A
   =\frac{\tr B_u}{\det B_u}
   =\frac{\Box_{\dS}u+2u}{\det B_u}.
\]
Thus \(H=0\) whenever \(u\) satisfies \eqref{eq:dS-equation}.

Conversely, let \(X\) be a timelike minimal immersion whose Gauss map is
locally invertible, and parametrize \(X\) by its Gauss map. Defining the
support function by
\(
u(p)=\langle X(p),p\rangle,
\)
the decomposition of \(X\) into its normal and tangential components yields
\[
X=up+\nabla u;
\]
this is the pseudo-Riemannian support-function formula of Reznikov
\cite{Reznikov1992}. Hence,
\(
dX=\Hess u+u\Id=B_u.
\)
Since \(X\) is minimal,
\(
0=\tr(B_u^{-1})
 =\frac{\tr B_u}{\det B_u},
\)
and therefore
\(
\Box_{\dS}u+2u=0.
\)
\end{proof}

Unlike in the elliptic situation, the singular set of $X_u$ need not be discrete.
Solutions of a hyperbolic equation can have support tensors that degenerate along curves; this agrees with the richer singularity theory of generalized timelike
minimal surfaces.

\subsection{Boundary quadrics and constant angle}

Let $\Gamma\subset\dS$ be a non-null curve.  Choose a unit conormal $\nu$ and
write
\[
 \langle\nu,\nu\rangle=\delta\in\{-1,1\}.
\]
Suppose
\begin{equation}\label{eq:dS-data}
 u=a,\quad \partial_\nu u=b
 \quad\text{on}\quad\Gamma.
\end{equation}
Then $\nabla u=\delta b\nu$ on $\Gamma$, and hence
\begin{equation*}\label{eq:dS-boundary-norm}
 \langle X_u,X_u\rangle=a^2+\delta b^2,
 \quad \langle X_u,p\rangle=a.
\end{equation*}

The next result shows that the prescribed Cauchy data imply the capillary boundary condition.

\begin{proposition}\label{prop:dS-capillary}
Let $u$ solve \eqref{eq:dS-equation} and satisfy \eqref{eq:dS-data} on a
noncharacteristic curve $\Gamma$.  Set
\[
 \rho=a^2+\delta b^2.
\]
Where $X_u$ is regular along $\Gamma$, its boundary lies on
\[
 \mathcal Q_\rho=\{x\in\Lthree:\langle x,x\rangle=\rho\}.
\]
If $\rho\neq0$, this is a nondegenerate totally umbilical quadric and $X_u$
meets it at constant Lorentzian contact angle.  If $\rho=0$, the boundary lies
on the light cone.
\end{proposition}

Thus the causal type of the conormal matters.  If $\Gamma$ is timelike, then
$\nu$ is spacelike and $\rho=a^2+b^2>0$.  If $\Gamma$ is spacelike, then
$\nu$ is timelike and $\rho=a^2-b^2$, allowing all three causal types.

\subsection{Local flexibility}

The equation \eqref{eq:dS-equation} is hyperbolic.  Constant
Dirichlet and conormal data along a noncharacteristic curve are therefore
Cauchy data, not elliptic overdetermination.
\medskip

The next result demonstrates the failure of local rigidity for the corresponding hyperbolic PDE.

\begin{theorem}\label{thm:dS-flexibility}
Let $\Gamma\subset\dS$ be an analytic noncharacteristic curve and let $a,b$ be
arbitrary real constants.  There exists a unique analytic solution of
\[
 \Box u+2u=0,
 \quad u|_\Gamma=a,
 \quad \partial_\nu u|_\Gamma=b
\]
in a neighborhood of $\Gamma$.  Consequently, the existence of constant
Cauchy data does not force $\Gamma$ to be a de Sitter circle, a geodesic, or an
orbit of a one-parameter isometry group.
\end{theorem}

\begin{proof}
Fix \(p\in\Gamma\). Since \(\Gamma\) is analytic and noncharacteristic, it is
non-null. Hence there exist analytic Gaussian coordinates \((s,t)\) in a
neighborhood \(U\) of \(p\) such that
\[
\Gamma\cap U=\{t=0\},
\quad
\partial_t|_{t=0}=\nu,
\quad
g_{st}=0,
\quad
g_{tt}=\delta,
\]
where
\(
\delta=\langle\nu,\nu\rangle\in\{-1,1\}
\)
(see \cite[Chapter~3 and Chapter~5, Section~1]{ONeill1983}). In these coordinates, the de Sitter metric has the form
\[
g=\delta\,dt^2-\delta\,\rho(s,t)^2\,ds^2
\]
for some positive analytic function \(\rho\), and so
\[
\Box u
 =
\delta\,u_{tt}
-\frac{\delta}{\rho^2}u_{ss}
+\delta\,\frac{\rho_t}{\rho}u_t
+\delta\,\frac{\rho_s}{\rho^3}u_s.
\]
Thus the equation \(\Box u+2u=0\) is equivalent to
\[
u_{tt}
=
\frac{1}{\rho^2}u_{ss}
-\frac{\rho_t}{\rho}u_t
-\frac{\rho_s}{\rho^3}u_s
-2\delta u.
\]
In particular, the equation is solved for the highest-order derivative
transverse to \(\Gamma\). This is precisely the role of the
noncharacteristic assumption.

Since \(\partial_t=\nu\) along \(t=0\), the prescribed conditions become
\[
u(s,0)=a,
\quad
u_t(s,0)=b.
\]
They are analytic Cauchy data, and all the coefficients in the equation are
analytic. The Cauchy--Kowalevski theorem
\cite{Hormander1990,John1982} therefore provides a unique analytic solution
in a neighborhood of \(p\). Uniqueness implies that the solutions obtained in
overlapping coordinate neighborhoods agree, producing a unique analytic
solution in a neighborhood of \(\Gamma\).
\end{proof}

\begin{corollary}\label{cor:dS-capillary-flexibility}
If, in addition, $\det B_u$ is nonzero at a point of $\Gamma$, then, after possibly shrinking the neighborhood, the solution given by Theorem~\ref{thm:dS-flexibility} generates a timelike minimal surface whose boundary lies on $\mathcal Q_\rho$ and has constant contact angle, provided $\rho\neq0$.
\end{corollary}

This is precisely where the de Sitter theory separates from the
spherical and hyperbolic theories.  The Hopf-differential maximum argument for
a disk is replaced by propagation from Cauchy data, and the boundary curve is
not selected by the PDE.

\subsection{Global rotational rigidity}

We use the following global coordinates on $\dS$:
\[
 p(t,\theta)=(\cosh t\cos\theta,\cosh t\sin\theta,\sinh t),
\]
with respect to which
\begin{equation*}\label{eq:dS-metric}
 g=-dt^2+\cosh^2t\,d\theta^2
\end{equation*}
and
\begin{equation}\label{eq:dS-box}
 \Box u=-u_{tt}-\tanh t\,u_t+\frac1{\cosh^2t}u_{\theta\theta}.
\end{equation}

In the next result, we establish a rigidity phenomenon in the rotationally symmetric setting.
\begin{theorem}\label{thm:dS-rotational}
Let $u$ be a smooth solution of $\Box u+2u=0$ on a globally hyperbolic strip
containing the Cauchy circle $\Sigma_{t_0}=\{t=t_0\}$.  If
\[
 u|_{\Sigma_{t_0}}=a,
 \quad u_t|_{\Sigma_{t_0}}=b
\]
are constant, then $u$ is rotationally invariant:
\[
 u(t,\theta)=f(t).
\]
The function $f$ satisfies
\begin{equation*}\label{eq:dS-radial-ode}
 f''+\tanh t\,f'-2f=0.
\end{equation*}
\end{theorem}

\begin{proof}
For each \(\phi\in\mathbb R\), the map
\(
R_\phi(t,\theta)=(t,\theta+\phi)
\)
is an isometry of $\dS$, preserves the Cauchy circle
\(\Sigma_{t_0}\), and fixes its unit normal. Hence the Klein--Gordon
operator commutes with \(R_\phi\):
\[
(\Box+2)(u\circ R_\phi)
 =
\bigl((\Box+2)u\bigr)\circ R_\phi=0.
\]
Moreover, since the Cauchy data of \(u\) are constant,
\[
(u\circ R_\phi)|_{\Sigma_{t_0}}=a,
\quad
\partial_t(u\circ R_\phi)|_{\Sigma_{t_0}}=b.
\]
Thus \(u\) and \(u\circ R_\phi\) solve the same Cauchy problem. Uniqueness
for normally hyperbolic equations on globally hyperbolic Lorentzian
manifolds
\cite[Theorem~3.2.11]{BarGinouxPfaeffle2007} gives
\[
u\circ R_\phi=u
\]
throughout the strip. Since this holds for every \(\phi\), \(u\) is independent
of \(\theta\), so \(u(t,\theta)=f(t)\). Substituting this expression into
\eqref{eq:dS-box} gives
\(
-f''-\tanh t\,f'+2f=0,
\)
which concludes the proof.
\end{proof}

The elementary solution $f(t)=\sinh t$ corresponds to a constant support map.
More generally, the ODE has a two-dimensional solution space, and regularity
of the associated timelike minimal immersion is governed by
$\det(\Hess f+f\Id)$.

\begin{remark}
If
\(
u(t,\theta)=\displaystyle\sum_{m\in\mathbb Z}f_m(t)e^{im\theta}
\)
is a solution of \eqref{eq:dS-box}, then every Fourier coefficient satisfies
\begin{equation*}
f_m''+\tanh t\,f_m'
+\left(\frac{m^2}{\cosh^2t}-2\right)f_m=0.
\end{equation*}
Hence constant Cauchy data on $\Sigma_{t_0}$ force $f_m\equiv0$ for every $m\neq0$, just as in the elliptic case of Theorem~\ref{thm:H-rigidity}.
This, however, holds only along the fixed circle $\Sigma_{t_0}$. By Theorem~\ref{thm:dS-flexibility}, constant Cauchy data can in fact be prescribed on every noncharacteristic curve in $\dS$. Rigidity of the domain is thus lost precisely at the point where the elliptic argument relied on fixing the boundary curve.
\end{remark}

\section*{Funding}
The authors were partially supported by the Brazilian National Council for Scientific and Technological Development (CNPq), Brazil [Grants 402563/2023-9 and 304381/2026-8 to M.B.; 409513/2023-7 and E:60030.0000000040/2026 FAPEAL/CNPq to I.D.]. Both authors were also supported by the Coordination for the Improvement of Higher Education Personnel (CAPES), Finance Code 001.

\section*{Data availability}
Data availability is not applicable to this article, as no data sets were generated or analyzed in the course of this research.

\section*{Conflict of interest}
The authors declare that they have no conflict of interest related to this article.

\end{document}